\documentclass[preprint,12pt]{elsarticle}
\usepackage[left=3cm,right=3cm,top=3cm,bottom=3cm,a4paper]{geometry}

\usepackage{amssymb,amsmath,amsthm,latexsym}
\usepackage{psfrag}
\usepackage{graphicx}

\theoremstyle{plain}
\newtheorem{Thm}{Theorem}[section]
\newtheorem{Lem}[Thm]{Lemma}
\newtheorem{Prop}[Thm]{Proposition}
\newtheorem{Cor}[Thm]{Corollary}

\theoremstyle{definition}
\newtheorem{Rem}[Thm]{Remark}

\newtheorem*{Defi}{Definition}

\journal{arXiv}

\begin{document}

\begin{frontmatter}

\title{Competitively tight graphs}

\author[label1,label5]{Suh-Ryung KIM}
\author[label2]{Jung Yeun LEE}
\author[label3,label6]{Boram PARK}
\author[label4,label7]{Yoshio SANO \corref{cor1}}

\address[label1]{Department of Mathematics Education,
Seoul National University, Seoul 151-742, Korea}
\address[label2]{National Institute for Mathematical Sciences,
Daejeon 305-390, Korea}
\address[label3]{DIMACS, Rutgers University,
Piscataway, NJ 08854, United States}
\address[label4]{National Institute of Informatics, Tokyo 101-8430, Japan}
\fntext[label5]{This research was supported
by Basic Science Research Program
through the National Research Foundation of Korea (NRF) funded
by the Ministry of Education, Science and Technology (700-20110036).}
\fntext[label6]{This work was supported
by National Research Foundation of Korea Grant funded
by the Korean Government, the Ministry of Education, Science and Technology
(NRF-2011-357-C00004).}
\fntext[label7]{The author was supported
by JSPS Research Fellowships for Young Scientists.}

\cortext[cor1]{Corresponding author. {\it E-mail address:}
{\tt sano@nii.ac.jp}; {\tt y.sano.math@gmail.com}}


\begin{abstract}
The competition graph of a digraph $D$ is a (simple undirected) graph
which has the same vertex set as $D$
and has an edge between two distinct vertices $x$ and $y$ if and only if
there exists a vertex $v$ in $D$ such that
$(x,v)$ and $(y,v)$ are arcs of $D$.
For any graph $G$, $G$ together with sufficiently many isolated vertices
is the competition graph of some acyclic digraph.
The competition number $k(G)$ of a graph $G$
is the smallest number of such isolated vertices.
Computing the competition number of a graph 
is an NP-hard problem in general 
and has been one of the important research problems
in the study of competition graphs. 
Opsut~[1982] showed that the competition number of a graph $G$ 
is related to the edge clique cover number $\theta_E(G)$ of the graph $G$ 
via $\theta_E(G)-|V(G)|+2 \leq k(G) \leq \theta_E(G)$. 
We first show that for any positive integer $m$ satisfying 
$2 \leq m \leq |V(G)|$, 
there exists a graph $G$ with $k(G)=\theta_E(G)-|V(G)|+m$
and characterize a graph $G$ satisfying $k(G)=\theta_E(G)$.
We then focus on what we call \emph{competitively tight graphs} $G$
which satisfy the lower bound, i.e., $k(G)=\theta_E(G)-|V(G)|+2$.
We completely characterize the competitively tight graphs
having at most two triangles.
In addition, we provide a new upper bound for the competition number 
of a graph from which we derive a sufficient condition and 
a necessary condition for a graph to be competitively tight. 
\end{abstract}


\end{frontmatter}

\section{Introduction}

Let $D$ be a digraph.
The \emph{competition graph} of $D$, denoted by $C(D)$,
is the (simple undirected) graph which
has the same vertex set as $D$
and has an edge between two distinct vertices $x$ and $y$
if and only if
there exists a vertex $v$ in $D$ such that $(x,v)$ and $(y,v)$
are arcs of $D$.
The notion of competition graph is due to Cohen~\cite{co}.
For any graph $G$, $G$ together with sufficiently many isolated
vertices is the competition graph of an acyclic digraph.
From this observation, Roberts \cite{cn}
defined the \emph{competition number} $k(G)$ of
a graph $G$ to be the minimum integer $k$ such that $G$
together with $k$ isolated vertices is 
the competition graph of an acyclic digraph.

It does not seem to be easy in general 
to compute the competition number $k(G)$ for a given graph $G$,
as Opsut~\cite{op} showed that the computation of the 
competition number of a graph is an NP-hard problem. 
To compute exact values or give bounds 
for the competition numbers of graphs 
has been one of the foremost problems 
in the study of competition graphs 
(see \cite{kimsu} for a survey). 

There is a well-known upper and lower bound for 
the competition numbers of arbitrary graphs 
due to Opsut \cite{op}.
A subset $S \subseteq V(G)$ of the vertex set of a graph $G$
is called a \emph{clique} of $G$ if the subgraph $G[S]$
of $G$ induced by $S$ is a complete graph.
For a clique $S$ of a graph $G$ and an edge $e$ of $G$,
we say \emph{$e$ is covered by $S$}
if both of the endvertices of $e$ are contained in $S$.
An \emph{edge clique cover} of a graph $G$
is a family of cliques of $G$ such that
each edge of $G$ is covered by some clique in the family.
The \emph{edge clique cover number} of a graph $G$, denoted $\theta_E(G)$,
is the minimum size of an edge clique cover of $G$.
Opsut showed the following. 

\begin{Thm}[{\cite[Propositions 5 and 7]{op}}]\label{thm:O-UB}
For any graph $G$, $\theta_E(G)-|V(G)|+2 \leq k(G) \leq \theta_E(G)$.
\end{Thm}

\noindent 
We note that the upper bound in Theorem~\ref{thm:O-UB} 
can be rewritten as 
$\theta_E(G)-|V(G)|+|V(G)|$, 
which leads us to ask:
For any integers $m$, $n$ satisfying $2 \leq m \leq n$,
does there exist a graph $G$ 
such that $|V(G)|=n$ and $k(G)=\theta_E(G)-n+m$?
The answer is yes by the following proposition:

\begin{Prop}
For any integers $m$ and $n$ satisfying $2 \leq m \leq n$,
there exists a graph $G$
such that $|V(G)|=n$ and $k(G)=\theta_E(G)-n+m$.
\end{Prop}

\begin{proof}
Let $V = \{ v_1, \ldots, v_n \}$ and $G$ be a graph on $V$
with the edge set consisting of the edges of the path
$v_1v_2 \cdots v_{n-m+1}$ and the edges of the complete graph
with vertex set $\{v_{n-m+1}, \ldots, v_n\}$.
Since $G$ is chordal, $k(G)=1$ (see \cite{cn}).
It is easy to see that $\theta_E(G)=(n-m)+1$.
Thus $k(G)=\theta_E(G)-n+m$.
\end{proof}

The upper bound in Theorem~\ref{thm:O-UB} is obtained
by only very special graphs as the following proposition states.

\begin{Prop}
Let $G$ be a graph with $n$ vertices. Then
the equality
$k(G)=\theta_E(G)$ holds if and only if
$G$ is the complete graph $K_n$  or
$G$ is the edgeless graph $I_n$.
\end{Prop}

\begin{proof}
If $G=K_n$, then $k(G)=1=\theta_E(G)$.
If $G=I_n$, then $k(G)=0=\theta_E(G)$.
Suppose that $G \neq K_n$ and $G \neq I_n$.
There exists an edge clique cover
$\{S_1, \ldots, S_r \}$ of $G$,
where $r=\theta_E(G)$.
Note that $r \geq 1$ since $G \neq I_n$,
and there exists a vertex $v \in V(G) \setminus S_1$ 
since $G \neq K_n$.
Take vertices $z_2, \ldots, z_r$ not in $G$. 
We define a digraph $D$ by
$V(D):=V(G) \cup I^*$ and
$A(D):=\{(x,v) : x \in S_1\} \cup A^*$, where $I^*$ and $A^*$ 
are defined as follows. If $r=1$, then $I^*=\emptyset$. 
Otherwise, $I^*=\{z_2, \ldots, z_r\}$.
If $r=1$, then $A^*=\emptyset$. 
Otherwise, $A^*=\cup_{i=2}^r \{(x,z_i) : x \in S_i \}$.
Thus we have $k(G) \leq r-1$, which implies that
$k(G) \neq \theta_E(G)$.
Hence the claim is true.
\end{proof}

Then it is natural to ask:
By which graphs is the lower bound in Theorem~\ref{thm:O-UB} achieved?
To answer this question, we introduce the following notion.

\begin{Defi}
{\rm
A graph $G$ is said to be 
\emph{competitively tight}
if it satisfies $k(G)=\theta_E(G)-|V(G)|+2$.
}
\end{Defi}

We can show by the following result that any triangle-free graph $G$
satisfying $|E(G)| \geq |V(G)|-1$ is competitively tight.

\begin{Thm}[{\cite[Theorem 8]{triangle-free}}]\label{thm:triangle-free}
If a graph $G$ is triangle-free, then
\[
k(G)=
\left\{
\begin{array}{ll}
0 & \text{ if } |V(G)|=1; \\
\max\{1, |E(G)|-|V(G)|+2 \}
& \text{ if } G \text{ has no isolated vertices}; \\
\max\{0, |E(G)|-|V(G)|+2 \} & \text{ otherwise.} \\
\end{array}
\right.
\]
\end{Thm}

\noindent
By this theorem, 
we know that a triangle-free graph $G$ without isolated vertices
satisfying $|E(G)| \geq |V(G)|-1 >0$ has the
competition number $|E(G)|-|V(G)|+2$.  
We also know that a triangle-free graph $G$ 
with isolated vertices satisfying $|E(G)| \geq |V(G)|-2 \geq 0$ 
has the competition number $|E(G)|-|V(G)|+2$. 
Since a triangle-free graph $G$ satisfies $\theta_E(G)=|E(G)|$,
the lower bound in Theorem~\ref{thm:O-UB} is achieved
by any triangle-free graph $G$ without isolated vertices
satisfying $|E(G)| \geq |V(G)|-1$ or any triangle-free graph $G$
with isolated vertices satisfying $|E(G)| \geq |V(G)|-2$.
Conversely, if a graph $G$ satisfies $|V(G)|=1$, then $|E(G)|-|V(G)|+2=1$
while $k(G)=0$. If a triangle-free graph $G$ has no isolated vertices
and $|E(G)| \leq |V(G)|-2$, then it cannot be competitively tight
since $k(G)=1$.
If a triangle-free graph $G$ has isolated vertices and $|E(G)|\leq |V(G)|-3$,
then it cannot be competitively tight since $k(G)$ is nonnegative.
Thus the competitively tight triangle-free graphs can be characterized
as follows:

\begin{Cor}\label{cor:TF}
A triangle-free graph $G$ is competitively tight if and only if
$|V(G)| \geq 2$ and $G$ satisfies one of the following:
\begin{itemize}
\item[{\rm (i)}]
$G$ has no isolated vertices and $|E(G)| \geq |V(G)|-1$;
\item[{\rm (ii)}]
$G$ has isolated vertices and $|E(G)| \geq |V(G)|-2$.
\end{itemize}
\end{Cor}

\section{Competitively tight graphs}

We begin this section by presenting simple but useful results 
which show how to obtain competitively tight graphs from existing ones.

\begin{Lem}\label{lem-2}
Let $G$ be a graph and $t$ be a nonnegative integer.
Then we have $k(G \cup I_t) \geq k(G)-t$, and
$k(G \cup I_t)=k(G)-t$ holds
if and only if $ 0 \leq t \leq k(G)$.
\end{Lem}

\begin{proof}
Suppose that $0 \leq t \leq k(G)$.
Let $k=k(G)$. Then $G \cup I_k=(G \cup I_t) \cup I_{k-t}$
is the competition
graph of an acyclic digraph.
Thus we have $k(G \cup I_t) \leq k-t=k(G)-t$.
To show that $k(G \cup I_t) \geq k(G)-t$, let $k'=k(G \cup I_t)$.
Then $(G \cup I_t) \cup I_{k'}=G \cup I_{t+k'}$
is the competition
graph of an acyclic digraph.
Thus we have $k(G) \leq t+k'$,
i.e. $k(G)-t \leq k'=k(G \cup I_t)$.
Hence we have $k(G \cup I_t)=k(G)-t$.
If $ k(G) <t$, then we have $k(G)-t < 0 \leq k(G \cup I_t)$.
Hence the lemma is true.
\end{proof}

\begin{Prop}\label{prop:isolated}
Suppose that a graph $G$ is competitively tight.
Let $t$ be an integer such that $ 0 \leq t \leq k(G)$.
Then the graph $G \cup I_t$ is competitively tight.
\end{Prop}

\begin{proof}
Since $G$ is competitively tight,
$k(G)=\theta_E(G)-|V(G)|+2$.
Since $ 0 \leq t \leq k(G)$,
by Lemma \ref{lem-2}, $k(G \cup I_t)=k(G)-t$ holds.
In addition, $\theta_E(G \cup I_t)=\theta_E(G)$ holds.
Thus we have $k(G \cup I_t)=\theta_E(G)-|V(G)|+2-t
=\theta_E(G \cup I_t)-|V(G \cup I_t)|+2$.
Hence $G \cup I_t$ is competitively tight.
\end{proof}

\begin{Prop}\label{prop:pendant}
Suppose that a graph $G$ $(\neq K_2)$ is competitively tight
and that $G$ has a vertex $v$ which is isolated or pendant.
Then the graph $G - v$ obtained from $G$
by deleting $v$ is competitively tight.
\end{Prop}

\begin{proof}
Since $G$ is competitively tight,
$k(G)=\theta_E(G)-|V(G)|+2$.
First,
suppose that $G$ has an isolated vertex $v$.
Since $v$ is isolated,
$\theta_E(G - v )=\theta_E(G)$.
By Lemma \ref{lem-2}, $k(G) \geq k(G -v)-1$.
So we have
$k(G - v ) \leq k(G)+1
=\theta_E(G)-|V(G)|+2+1
=\theta_E(G - v)-|V(G - v)|+2$.
By Theorem~\ref{thm:O-UB}, $k(G - v )
\geq \theta_E(G- v )-|V(G- v )|+2$.
Hence we have $k(G - v )=\theta_E(G - v )
-|V(G - v )|+2$, i.e.,
$G - v$ is competitively tight.
Second,
suppose that $G$ has a pendant vertex $v$.
Note that $\theta_E(G - v)=\theta_E(G)-1$ and
$|V(G - v)|=|V(G)|-1$.
Obviously $k(G - v)=k(G)$ when $G \neq K_2$.
Hence we have $k(G - v)=\theta_E(G - v)
-|V(G - v)|+2$, i.e.,
$G - v$ is competitively tight.
\end{proof}

Due to Propositions~\ref{prop:isolated} and \ref{prop:pendant},
when we consider competitively tight graphs $G$,
we may assume that the minimum degree of $G$ is at least two.

We now begin the examination of competitively tight graphs 
which are not triangle-free. 
To this end, we recall several results of Kim and Roberts~\cite{KRTrian} 
which determine the competition numbers 
of various graphs with triangles of varying complexity. 
They found the competition number 
of a graph with exactly one triangle 
as the following theorem illustrates.

\begin{Thm}[{\cite[Corollary~7]{KRTrian}}]\label{onetrian}
Suppose that a graph $G$ is connected and has exactly one triangle.
Then 
\[
k(G)=
\left\{
\begin{array}{ll}
|E(G)|-|V(G)| & \text{ if $G$ has a cycle of length at least four}; \\
|E(G)|-|V(G)|+1 & \text{ otherwise}.
\end{array}
\right.
\]
\end{Thm}

Let $G$ be a connected graph with exactly one triangle.
Then $\theta_E(G)=|E(G)|-2$. If $G$ has a cycle of length at least four,
then $k(G)=\theta_E(G)-|V(G)|+2$ by Theorem~\ref{onetrian}.
Otherwise, $k(G)=\theta_E(G)-|V(G)|+3$ by the same theorem.
Now we have a characterization for the connected competitively tight graphs
with exactly one triangle.

\begin{Prop}\label{ctonetr}
A connected graph with exactly one triangle is competitively tight
if and only if it has a cycle of length at least four. 
\end{Prop}

Kim and Roberts~\cite{KRTrian} also determined the competition number
of a graph with exactly two triangles. 
To do so, they defined ${\rm VC}(G)$ for a graph $G$ as
\[
{\rm VC}(G) := \{v \in V(G) : v
\mbox{ is a vertex on a cycle of } G \}.
\]
Let $\mathcal{G}_1$ (resp. $\mathcal{G}_2$) be the family of graphs
that can be obtained
from Graph I (resp. one of the Graphs II-V)
in Figure~\ref{triangles}
by subdividing edges except those on triangles.

\begin{figure}
\begin{center}
\psfrag{A}{I}
\psfrag{B}{II}
\psfrag{C}{III}
\psfrag{D}{IV}
\psfrag{E}{V}
\includegraphics[scale=0.95]{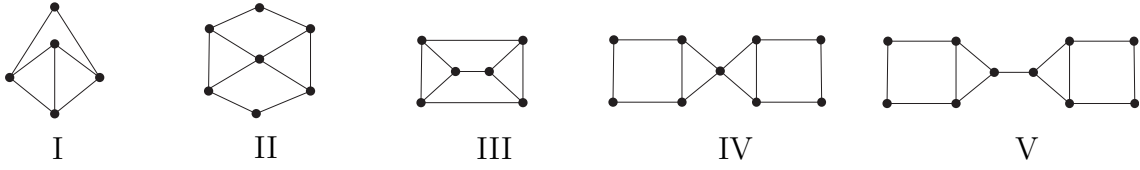}
\caption{Graphs with exactly two triangles}
\label{triangles}
\end{center}
\end{figure}

\begin{Thm}[{\cite[Theorem~9]{KRTrian}}]\label{thm:share}
Suppose that a connected graph $G$ has exactly two triangles
which share one of their edges. Then
\begin{itemize}
\item[{\rm (a)}]
$k(G)=|E(G)|-|V(G)|$ if $G$ is chordal or if
the subgraph induced by
${\rm VC}(G)$
is
in $\mathcal{G}_1$, and
\item[{\rm (b)}]
$k(G)=|E(G)|-|V(G)|-1$ otherwise.
\end{itemize}
\end{Thm}

\begin{Thm}[{\cite[Theorem~10]{KRTrian}}]\label{thm:notshare}
Suppose that a connected graph $G$ has exactly two triangles
which are edge-disjoint. Then
\begin{itemize}
\item[{\rm (a)}]
$k(G)=|E(G)|-|V(G)|$ if $G$ is chordal,
\item[{\rm (b)}]
$k(G)=|E(G)|-|V(G)|-1$ if $G$ has
exactly one cycle of length at least four
as an induced
subgraph
or if
the subgraph induced by
${\rm VC}(G)$
is
in $\mathcal{G}_1 \cup \mathcal{G}_2$, and
\item[{\rm (c)}]
$k(G)=|E(G)|-|V(G)|-2$ otherwise.
\end{itemize}
\end{Thm}

\noindent
From these two theorems, 
we may characterize the competitively tight graphs 
with exactly two triangles. 
A cycle of length at least four in a graph $G$ 
is called a \emph{hole} of $G$ 
if it is an induced subgraph of $G$. 
The number of holes of a graph is closely related to 
its competition number (see \cite{KLPS, MWW}).

\begin{Thm}\label{thm:twotri}
A connected graph $G$ with exactly two triangles is competitively tight
if and only if $G$ is not chordal and satisfies one of the following:
\begin{itemize}
\item[{\rm (i)}]
the two triangles share one of their edges
and
the subgraph induced by ${\rm VC}(G)$
is not in $\mathcal{G}_1$;
\item[{\rm (ii)}]
the two triangles are edge-disjoint,
$G$ contains at least two holes,
and
the subgraph induced by ${\rm VC}(G)$
is not in $\mathcal{G}_1 \cup \mathcal{G}_2$.
\end{itemize}
\end{Thm}

\begin{proof}
Let $\Delta_1$ and $\Delta_2$ be the two triangles of $G$.
If $\Delta_1$ and $\Delta_2$ share an edge,
then $\theta_E(G)=|E(G)|-3$.
If $\Delta_1$ and $\Delta_2$ are edge-disjoint,
then $\theta_E(G)=|E(G)|-4$.
First, we show the ``if" part.
If (i) holds, then $G$ satisfies
the hypothesis of Theorem~\ref{thm:share} (b)
and so $k(G)=|E(G)|-|V(G)|-1=\theta_E(G)-|V(G)|+2$.
If (ii) holds, then $G$ satisfies
the hypothesis of Theorem~\ref{thm:notshare} (c),
$k(G)=|E(G)|-|V(G)|-2=\theta_E(G)-|V(G)|+2$.
Second, we show the ``only if" part by contradiction.
Suppose that $G$ is chordal.
Then, by Theorems~\ref{thm:share} and \ref{thm:notshare},
$k(G)=|E(G)|-|V(G)|$ or $|E(G)|-|V(G)|-1$,
none of which equals $\theta_E(G)-|V(G)|+2$.
Thus if $G$ is chordal, then $G$ is not competitively tight.
Suppose that neither (i) nor (ii) holds.
We consider the case where $\Delta_1$ and $\Delta_2$ share an edge.
Then, by Theorem~\ref{thm:share} (a), ${\rm VC}(G)$
induces
a graph in $\mathcal{G}_1$ and so $k(G)=|E(G)|-|V(G)|$,
which does not equal $(|E(G)|-3)-|V(G)|+2=\theta_E(G)-|V(G)|+2$.
Now we consider the case where $\Delta_1$ and $\Delta_2$ are edge-disjoint.
Then $G$ contains exactly one hole or ${\rm VC}(G)$
induces
a graph in $\mathcal{G}_1 \cup \mathcal{G}_2$.
By Theorem~\ref{thm:notshare} (b),
$k(G)=|E(G)|-|V(G)|-1 \neq (|E(G)|-4)-|V(G)|+2=\theta_E(G)-|V(G)|+2$.
\end{proof}

It does not seem to be easy to characterize the competitively tight graph
with exactly three triangles.
Yet, we can show that there exists a competitively tight graph
with exactly $n$ triangles for each nonnegative integer $n$. 
We first give a new upper bound which improves the one given
in Theorem~\ref{thm:O-UB}.  
Let $G$ be a graph and $F$ be a subset of the edge set of $G$. 
We denote by $\theta_E(F;G)$ 
the minimum size of a family $\mathcal{S}$ of cliques of $G$ 
such that each edge in $F$ is covered by some clique in the family 
$\mathcal{S}$ 
(cf. \cite{Sano}). 
We also need to introduce some notations.
For a graph $G$, we define 
\begin{eqnarray*}
E_{\triangle}(G) &:=& \{e \in E(G) : e
\text{ is contained in a triangle in } G \}, \\
\overline{E}_{\triangle}(G) &:=& \{e \in E(G) : e
\text{ is not contained in any triangle in } G \}.
\end{eqnarray*}
Note that
$E_{\triangle}(G) \cup \overline{E}_{\triangle}(G) = E(G)$ and
$E_{\triangle}(G) \cap \overline{E}_{\triangle}(G) = \emptyset$,
and we can easily check the following lemma from the definitions.

\begin{Lem}\label{lem:ECC}
For any graph $G$,
$\theta_E(G) = \theta_E(E_{\triangle}(G); G)
+ |\overline{E}_{\triangle}(G)|$.
\end{Lem}

Now we present a new upper bound for the competition number of a graph.

\begin{Thm}\label{thm:main}
For any graph $G$,
\[
k(G) \leq \theta_E(E_{\triangle}(G); G)
+ \max \{ \min \{ 1,|\overline{E}_{\triangle}(G)| \},
|\overline{E}_{\triangle}(G)| - |V(G)| + 2\}.
\]
\end{Thm}

\begin{proof}
Let $H$ be the graph obtained from $G$
by deleting the edges in $E_{\triangle}(G)$, i.e.,
$H:=G-E_{\triangle}(G)$. Then $H$ is triangle-free and so,
by Theorem~\ref{thm:triangle-free},
\[
k(H) \leq \max \{ \min \{ 1,|E(H)| \}, |E(H)| - |V(H)| + 2\}.
\]
Since $V(H)=V(G)$, $E(H)=\overline{E}_{\triangle}(G)$,
the above inequality is equivalent to
\[
k(H) \leq \max \{ \min \{ 1,|\overline{E}_{\triangle}(G)| \},
|\overline{E}_{\triangle}(G)| - |V(G)| + 2\}.
\]
Let $D^-$ be an acyclic digraph such that
$C(D^-)=H \cup I_{k(H)}$.
Let $\mathcal{S}$ be 
a family of cliques of $G$ of size $\theta_E(E_{\triangle}(G); G)$ 
such that each edge in $E_{\triangle}(G)$ 
is covered by some clique in $\mathcal{S}$. 
We define a digraph $D$ by
\[
V(D) := V(D^-) \cup \{z_S : S \in \mathcal{S} \}
\quad \text{ and } \quad
A(D) := A(D^-) \cup \bigcup_{S \in \mathcal{S}} \{(v, z_S) : v \in S \}.
\]
Then $D$ is acyclic, and
$C(D) = G \cup I_{k(H)} \cup \{z_S : S \in \mathcal{S} \}$.
Therefore
\[
k(G) \leq |\mathcal{S}| + k(H)
\leq \theta_E(E_{\triangle}(G); G)
+ \max \{ \min \{ 1,|\overline{E}_{\triangle}(G)| \},
|\overline{E}_{\triangle}(G)| - |V(G)| + 2\}.
\]
Thus the theorem is true.
\end{proof}

\begin{Rem}
The upper bound given in Theorem \ref{thm:main}
is always better than the upper bound in Theorem \ref{thm:O-UB}.
Indeed, the following inequality holds for any graph $G$
\[
\theta_E(E_{\triangle}(G); G)
+ \max \{ \min \{ 1,|\overline{E}_{\triangle}(G)| \},
|\overline{E}_{\triangle}(G)| - |V(G)| + 2\}
\leq \theta_E(G).
\]
\end{Rem}

\begin{proof}
If $|\overline{E}_{\triangle}(G)|=0$,
then the left hand side of the above inequality is
equal to
$\theta_E(E_{\triangle}(G); G)$
which is less than or equal to $\theta_E(G)$.
Now suppose that $|\overline{E}_{\triangle}(G)| \geq 1$.
Then $\min\{1,|\overline{E}_\triangle(G)|\}=1$
and the left hand side is equal to
$\theta_E(E_{\triangle}(G); G) +1$
or
$\theta_E(E_{\triangle}(G); G)
+ |\overline{E}_{\triangle}(G)| - |V(G)| + 2$.
In addition, $|V(G)| \geq 2$.
Since $|\overline{E}_{\triangle}(G)| \geq 1$
and $|V(G)| \geq 2$,
both $\theta_E(E_{\triangle}(G); G) +1$
and
$\theta_E(E_{\triangle}(G); G)
+ |\overline{E}_{\triangle}(G)| - |V(G)| + 2$
are less than or equal to $ \theta_E(E_{\triangle}(G); G)
+ |\overline{E}_{\triangle}(G)|$.
Thus, the inequality holds by Lemma \ref{lem:ECC}.
\end{proof}

As a corollary of Theorem~\ref{thm:main},
we obtain the following result
which gives a sufficient condition
for graphs to be competitively tight.

\begin{Cor} \label{sufficient}
If a graph $G$ satisfies
$|\overline{E}_{\triangle}(G)| \geq |V(G)| -i(G)- 1$ 
and $i(G) \leq k(G)$ 
where $i(G)$ is the number of isolated vertices of $G$,
then
$G$ is competitively tight.
\end{Cor}

\begin{proof}
Let $G'$ be the graph obtained by deleting the isolated vertices from $G$. 
Since $i(G) \leq k(G)$, it follows from Lemma~\ref{lem-2} that 
$k(G)=k(G')-i(G)$. 
Since $k(G) \geq 0$, 
we have $i(G) \leq k(G')$. 
Thus, by Proposition~\ref{prop:isolated}, it is sufficient 
to show that $G'$ is competitively tight.
Since $|\overline{E}_{\triangle}(G')|=|\overline{E}_{\triangle}(G)|$ 
and $|V(G')|=|V(G)|-i(G)$, 
we have $|\overline{E}_{\triangle}(G')| \geq |V(G')|-1$ 
and so $|\overline{E}_{\triangle}(G')| - |V(G')| + 2 \geq 1 \geq
\min \{ 1,|\overline{E}_{\triangle}(G')| \}$.
By Lemma \ref{lem:ECC} and Theorem \ref{thm:main},
$k(G') \leq
\theta_E(E_{\triangle}(G'); G')
+ |\overline{E}_{\triangle}(G')| - |V(G')| + 2
= \theta_E(G') -|V(G')| +2$.
By Theorem~\ref{thm:O-UB},
we obtain
$k(G') = \theta_E(G') -|V(G')| +2$.
\end{proof}

\begin{figure}
\begin{center}
\psfrag{1}{\footnotesize$v_1$}
\psfrag{2}{\footnotesize$v_2$}
\psfrag{3}{\footnotesize$v_3$}
\psfrag{4}{\footnotesize$v_4$}
\psfrag{5}{\footnotesize$v_5$}
\psfrag{6}{\footnotesize$v_6$}
\psfrag{7}{\footnotesize$v_7$}
\psfrag{8}{\footnotesize$v_8$}
\psfrag{9}{\footnotesize$v_9$}
\psfrag{0}{\footnotesize$v_{10}$}
\psfrag{a}{\footnotesize$v_{11}$}
\psfrag{b}{\footnotesize$v_{12}$}
\psfrag{c}{\footnotesize$v_{13}$}
\psfrag{d}{\footnotesize$v_{14}$}
\psfrag{e}{\footnotesize$v_{15}$}
\psfrag{f}{\footnotesize$v_{16}$}
\psfrag{g}{\footnotesize$v_{17}$}
\psfrag{h}{\footnotesize$v_{18}$}
\psfrag{19}{\footnotesize$v_{19}$}
\psfrag{20}{\footnotesize$v_{20}$}
\psfrag{21}{\footnotesize$v_{21}$}
\psfrag{22}{\footnotesize$v_{22}$}
\psfrag{23}{\footnotesize$v_{23}$}
\psfrag{24}{\footnotesize$v_{24}$}
\psfrag{25}{\footnotesize$v_{25}$}
\psfrag{26}{\footnotesize$v_{26}$}
\psfrag{27}{\footnotesize$v_{27}$}
\psfrag{28}{\footnotesize$v_{28}$}
\psfrag{29}{\footnotesize$v_{29}$}
\psfrag{30}{\footnotesize$v_{30}$}
\psfrag{31}{\footnotesize$v_{31}$}
\psfrag{32}{\footnotesize$v_{32}$}
\psfrag{33}{\footnotesize$v_{33}$}
\psfrag{34}{\footnotesize$v_{34}$}
\psfrag{35}{\footnotesize$v_{35}$}
\psfrag{36}{\footnotesize$v_{36}$}
\includegraphics{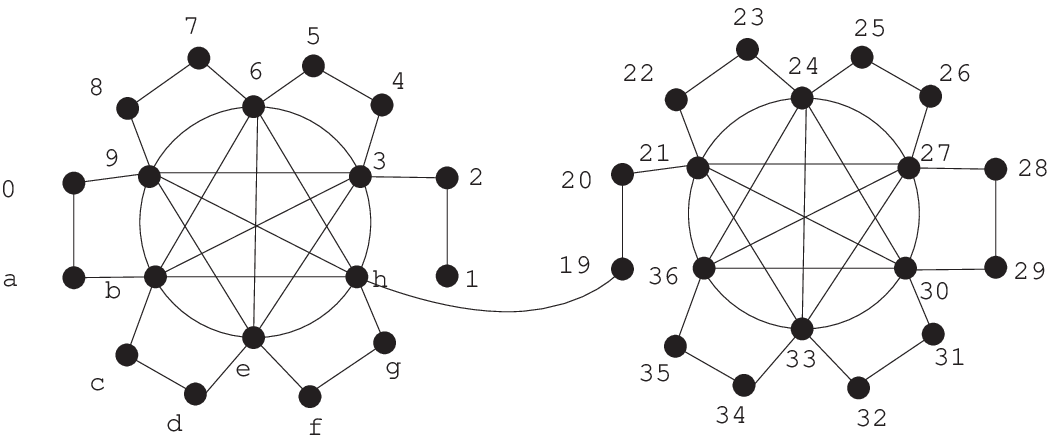}
\end{center}
\caption{$G_{6,2}$.}\label{fg}
\end{figure}

We present a family of graphs satisfying
the sufficient condition for a graph being competitively tight.
Let $t$ and $n$ be positive integers with $t \geq 3$.
Let $G_{t,n}$ be the connected graph defined by
\begin{eqnarray*}
V(G_{t,n}) &=& \{v_1, \ldots, v_{3tn} \}, \\
E(G_{t,n}) &=& \{v_i v_{i+1} : 1 \leq i \leq 3tn-1 \} 
\cup \bigcup_{m=0}^{n-1}\{v_{3tm+3i} v_{3tm+3j}
: 1 \leq i < j \leq t \}
\end{eqnarray*}
(see Figure~\ref{fg}).
It is easy to check that $\overline{E}_\triangle(G_{t,n})$
is the Hamilton path $v_1v_2 \ldots v_{3tn}$  of $G_{t,n}$
and so $|\overline{E}_\triangle(G_{t,n})|=|V(G)|-1$.

On the other hand,
each of the edges on the Hamilton path $v_1v_2 \cdots v_{3tn}$
forms a maximal clique. Other than those cliques,
$\{v_{3tm+3i}v_{3tm+3j} : 1 \leq i < j \leq t\}$
is a maximal clique for each $m$, $0 \leq m \leq n-1$.
It can easily be seen that these maximal cliques form
an edge clique cover whose size is minimum among
all edge clique covers of $G_{t,n}$,
which implies that $\theta_E(G_{t,n})=(3tn-1) + n$.
Thus, by Corollary~\ref{sufficient},
\[
k(G) =(3tn+n-1)-3tn+2=n+1.
\]

For any positive integer $n$,
let $G=G_{3,n}$.
Then $v_{9i+3} v_{9i+6} v_{9i+9}$ are the only triangles of $G$
$(0 \leq i \leq n-1)$ and so $G$ has exactly $n$ triangles.
As we have shown,
it holds that $k(G) =n+1= \theta_E(G)-|V(G)|+2$.
Hence $G$ is a competitively tight graph with exactly $n$ triangles.

It is also possible that a competitively tight graph
has a clique of any size:
For any positive integer $t$ with $t \geq 3$,
let $G = G_{t,1}$.
Then $S=\{v_{3i} \in V : 1 \leq i \leq t \}$ is a clique of size $t$ of $G$.
As we have shown,
it holds that $k(G) =2= \theta_E(G)-|V(G)|+2$.
Hence $G$ is competitively tight.

The following gives a necessary condition
for graphs to be competitively tight.

\begin{Prop}\label{prop:nec}
If a graph $G$ is competitively tight,
then $|\overline{E}_{\triangle}(G)|
\geq |V(G)| - \theta_E(E_{\triangle}(G); G) -2$.
\end{Prop}

\begin{proof}
Since $G$ is competitively tight,
$k(G)=\theta_E(G)-|V(G)|+2$ holds.
By Lemma \ref{lem:ECC},
we have
$\theta_E(E_{\triangle}(G); G)
+ |\overline{E}_{\triangle}(G)|-|V(G)|+2
= \theta_E(G)-|V(G)|+2 =k(G) \geq 0$.
Hence $|\overline{E}_{\triangle}(G)|
\geq |V(G)| - \theta_E(E_{\triangle}(G); G) -2$.
\end{proof}

It follows from Corollary~\ref{sufficient}
that any graph $G$ having exactly three triangles 
and having no isolated vertices is competitively tight
if it satisfies $|E(G)| \geq |V(G)|+8$.
To see why, note that $|E_\triangle(G)| \leq 9$.
Since $|\overline{E}_\triangle(G)|=|E(G)|-|E_\triangle(G)|$,
\[
|\overline{E}_\triangle(G)|=|E(G)|-|E_\triangle(G)|
\geq (|V(G)|+8)-9 \geq |V(G)|-i(G)-1
\]
and so, by Corollary~\ref{sufficient}, $G$ is competitively tight.

On the other hand, we know from Proposition~\ref{prop:nec}
that a graph $G$ having exactly three triangles is not competitively tight
if it satisfies $|E(G)| \leq |V(G)|+1$.
To show it, we first note that $\theta_E(E_\triangle(G);G)=3$
and $7 \leq |E_\triangle(G)|$.
Then
\[
|\overline{E}_\triangle(G)|=|E(G)|-|E_\triangle(G)|
\leq (|V(G)|+1)-7 \leq |V(G)|-\theta_E(E_\triangle(G);G)-3
\]
and so, by Proposition~\ref{prop:nec}, $G$ is not competitively tight.

\section{Further Study}

The lower bound given in Corollary \ref{sufficient} can be improved. 
To take a competitively tight graph which does not satisfy the condition 
of Corollary~\ref{sufficient}, 
let $n$ and $p$ be integers with $n \geq 7$ and 
$2 \leq p < \lfloor \tfrac{n}{3} \rfloor$. 
Let $G$ be the Cayley graph associated with 
$(\mathbb{Z} \slash n \mathbb{Z}, \{\pm 1, \pm 2, \ldots, \pm p\})$,
i.e., $G$ is the graph defined by 
\[
V(G) = \{v_i : i \in \mathbb{Z} \slash n \mathbb{Z} \} 
\quad \text{ and } \quad 
E(G)=\{v_i v_j : i-j \in \{\pm 1, \pm 2, \ldots, \pm p\}\}. 
\]
Then $|V(G)|=n$ and $|\overline{E}_{\triangle}(G)|=0$. 
Therefore, 
$|\overline{E}_{\triangle}(G)| < |V(G)| - 1$. 
As $i(G)=0$, $G$ does not satisfy 
the condition of Corollary \ref{sufficient}. 

Since any two of the edges in 
$\{v_i v_j : i-j \in \{\pm p\} \}$ are not
covered by the same clique in $G$, 
any edge clique cover of $G$ contains at least $n$ cliques. 
Therefore, $\theta_E(G) \geq n$.
Since $\theta_E(G) \leq n$ by {\cite[Lemma 2.4]{KPS}},
we have $\theta_E(G)=n$.
Note that $k(G)=2$ by {\cite[Theorem 1.3]{KPS}}.
Thus $G$ is competitively tight. 

Accordingly, we propose improving the lower bound 
given in Corollary \ref{sufficient} as a further study. 
In a similar vein, we suggest finding out 
whether or not the lower bound given 
in Proposition~\ref{prop:nec} is sharp.

\section*{Acknowledgments}

The authors 
are grateful to 
the anonymous referees 
for suggestions leading to improvements 
in the presentation of the results.



\begin{thebibliography}{99}

\bibitem{co}
J. E. Cohen:
Interval graphs and food webs: a finding and a problem,
RAND Corporation Document 17696-PR, Santa Monica, CA, 1968.

\bibitem{kimsu}
S. -R. Kim:
The competition number and its variants,
in \emph{Quo Vadis, Graph Theory?}
(J. Gimbel, J. W. Kennedy, and L. V. Quintas, eds.),
Annals of Discrete Mathematics \textbf{55},
North Holland B. V., Amsterdam, the Netherlands, (1993) 313-326.

\bibitem{triangle-free}
{S. -R. Kim}:
{The competition number of triangle-free graphs},
\emph{Congressus Numerantium} \textbf{110} (1995) 97--105.

\bibitem{KLPS}
{S. -R. Kim, J. Y. Lee, B. Park, and Y. Sano}:
{The competition number of a graph and the dimension of its hole space},
\emph{Applied Mathematics Letters}
\textbf{25} (2012) 638--642.

\bibitem{KPS}
{S. -R. Kim, B. Park, and Y. Sano}:
{The competition number of the complement of a cycle},
\emph{Discrete Applied Mathematics}, Available online (2011). 
\texttt{doi:10.1016/j.dam.2011.10.034}

\bibitem{KRTrian}
{S. -R. Kim and F. S. Roberts}:
{Competition numbers of graphs with a small number of triangles},
\emph{Discrete Applied Mathematics} \textbf{78} (1997) 153--162.

\bibitem{MWW}
{B. D. McKay, P. Schweitzer, and P. Schweitzer}:
{Competition numbers, quasi-line graphs and holes}, 
\emph{Preprint}. 
\texttt{arXiv:1110.2933v1}

\bibitem{op}
{R. J. Opsut}:
{On the computation of the competition number of a graph},
\emph{SIAM Journal on Algebraic and Discrete Methods}
\textbf{3} (1982) 420--428.

\bibitem{cn}
{F. S. Roberts}:
{Food webs, competition graphs, and the boxicity of ecological phase space},
\emph{Theory and Applications of Graphs (Proc. Internat. Conf.,
Western Mich. Univ., Kalamazoo, Mich., 1976)},
Lecture Notes in Mathematics \textbf{642}, Springer 
(1978) 477--490.

\bibitem{Sano}
{Y. Sano}:
{A generalization of Opsut's lower bounds for
the competition number of a graph},
\emph{Graphs and Combinatorics}, Published online (2012). 
\texttt{doi:10.1007/s00373-012-1188-5}

\end{thebibliography}
\end{document}